\documentclass[10pt]{amsart}
\makeatletter
\def\blfootnote{\gdef\@thefnmark{}\@footnotetext}
\makeatother
\usepackage{epsfig}
\usepackage{graphics}

\usepackage{amscd,amsthm,amsfonts}

\usepackage[colorlinks=true,
                    linkcolor=red,
                    urlcolor=red,
                    citecolor=red,
                    anchorcolor=red]{hyperref}
\usepackage{amssymb}
\usepackage{footnote}
\usepackage{multicol}
\usepackage[pagewise]{lineno}

\newtheorem{theorem}{Theorem}[section]

\newtheorem{lemma}[theorem]{Lemma}

\newtheorem{question}[theorem]{Question}

\theoremstyle{definition}

\newtheorem{remark}[theorem]{Remark}

\theoremstyle{remark}

\numberwithin{equation}{section}

\def\Mod{{\rm Mod}}
\def\HMod{{\rm HMod}}

\begin{document}
 \blfootnote{\textup{2000} \textit{Mathematics Subject Classification}:
57M99, 20F38}
\blfootnote {\textit{Keywords}: Symplectic $4$-manifolds,
Mapping class groups, Lefschetz fibrations, exotic manifolds}

\newenvironment{prooff}{\medskip \par \noindent {\it Proof}\ }{\hfill
$\square$ \medskip \par}
    \def\sqr#1#2{{\vcenter{\hrule height.#2pt
        \hbox{\vrule width.#2pt height#1pt \kern#1pt
            \vrule width.#2pt}\hrule height.#2pt}}}
    \def\square{\mathchoice\sqr67\sqr67\sqr{2.1}6\sqr{1.5}6}
\def\pf#1{\medskip \par \noindent {\it #1.}\ }
\def\endpf{\hfill $\square$ \medskip \par}
\def\demo#1{\medskip \par \noindent {\it #1.}\ }
\def\enddemo{\medskip \par}
\def\qed{~\hfill$\square$}

 \title[Small Lefschetz Fibrations on Simply-Connected $4$-Manifolds]
{Small Lefschetz Fibrations on Simply-Connected $4$-Manifolds}
\author[T{\"{u}}l\.{i}n Altun{\"{o}}z]{T{\"{u}}l\.{i}n Altun{\"{o}}z}
 \address{Department of Mathematics, Middle East Technical University,
  Ankara, Turkey}
  \email{atulin@metu.edu.tr}
\date{\today}
\maketitle
 \begin{abstract} We consider simply-connected $4$-manifolds admitting Lefschetz fibrations over the $2$-sphere. We explicitly construct nonhyperelliptic and hyperelliptic Lefschetz fibrations of genus $4$ on simply-connected $4$-manifolds which are exotic symplectic $4$-manifolds in the homeomorphism classes of $\mathbb{C} P^{2}\#8\overline{\mathbb{C} P^{2}}$ and $\mathbb{C} P^{2}\#9\overline{\mathbb{C} P^{2}}$, respectively. From these, we provide upper bounds for the minimal number of singular fibers of such fibrations. In addition, we prove that this number is equal to $18$ for $g=3$ when such fibrations are hyperelliptic. Moreover, we discuss these numbers for higher genera.
\end{abstract}

\section{Introduction}
\label{Sec:1}

Due to the pioneering works of Donaldson~\cite{d2} and Gomp~\cite{gs}, Lefschetz pencils and Lefschetz fibrations play an important role in studying $4$-manifold topology. Donaldson proved that every symplectic $4$-manifold, up to blow-ups, corresponds to a Lefschetz fibration with a finite number of singularities of a prescribed type, which provides a way to study combinatorially via a positive factorization of its monodromy if exists. Conversely, Gompf showed that any $4$ manifold admitting a genus-$g$ Lefschetz fibration is a symplectic $4$-manifold if $g\geq2$.

Every nontrivial Lefschetz fibration admits certain singular fibers associated to its monodromy. The number of its singular fibers  provides us important information about its total spaces such as the Euler characteristic, the signature, and so on. Since it has been known that the number of singular fibers in a Lefschetz fibration cannot be arbitrary, determining the minimal number of singular fibers in a Lefschetz fibration is one of the interesting problem to be investigated.

Let $N(g,h)$ be the minimal number of singular fibers in all nontrivial relatively minimal genus-$g$ Lefschetz fibrations of over the oriented closed surface of genus-$h$. For $h\geq1$, the exact value of the number $N(g,h)$ is almost known (except $N(g,1)$ for $g\geq3$ and $N_{2,2}$)~\cite{Ha,Ko,Ks,M,StY}. For $h=0$, it is known that $N(2,0)=7$ by Xiao's construction~\cite{X} and also by the existence of a relation among seven positive Dehn twists in the mapping class group of genus-$2$ surface with one boundary component obtained by Baykur and Korkmaz~\cite{bki}. However, the exact value of this number is not known for $h=0$ and $g\geq3$. The best known estimates for $N(g,0)$ : $N(g,0)\leq 2g+4$ if $g$ is even and $N(g,0)\leq 2g+10$ if $g$ is odd~\cite{Cadavid,dmp,k1}. If we consider hyperelliptic Lefschetz fibrations, i.e. their vanishing cycles are invariant under a hyperelliptic involution $\iota$ (see for instance Figure~\ref{TAB3}), then some results about this number are known. Let $M(g,h)$ denote the minimal number of singular fibers in all nontrivial genus-$g$ hyperelliptic Lefschetz fibrations over the $2$-sphere. Baykur and Korkmaz~\cite{bk1} constructed a hyperelliptic genus-$3$ Lefschetz fibration over the $2$-sphere with $12$ singular fibers and then they proved that $M(3,0)=12$. In~\cite{AT} the author proved the following result:\\
$(1)$ $M(4,0)=12$ and $M(6,0)=16$, \\
$(2)$ $M(8,0)=19$ or $20$ and $M(10,0)=23$ or $24$,\\
 $(3)$ $M(5,0)\geq15$, $M(7,0)\geq17$, and $M(9,0)\geq 24$. \\Moreover, when the total spaces of such fibrations are complex surfaces, she proved that it is equal to $2g+4$ if $g\geq4$ and even and it has a lower bound $2g+6$ if $g\geq7$ and odd. Thus, the exact value of the number $M(g,0)$ is not known (except for $g=3,4$ and $6$). Therefore, this question is also open for hyperelliptic Lefschetz fibrations.

Lefschetz fibration structures on various smooth $4$-manifolds with small numbers of singular fibers may provide us the existence of symplectic structures on $4$-manifolds in the homeomorphism classes of simply-connected $4$-manifolds with very small topology, which has been an interesting topic contains several construction techniques (e.g. ~\cite{a1,a3,bk,d1,f1,f2,f3,fmor,g1,kot,p1,p2,sz}). Recently, some authors have studied Lefschetz fibration structures to produce exotic $4$-manifolds (which are homeomorphic but not diffeomorphic to standard ones). Since it is natural to relate small  (as in small second homology) exotic $4$-manifolds to small (as in small number of positive Dehn twists) Lefschetz fibrations, it is of interest to find the minimal number of singular fibers in Lefschetz fibrations on simply-connected $4$-manifolds.

Let $N_g$ be the minimal number of singular fibers in all genus-$g$ Lefschetz fibrations on a simply-connected $4$-manifold over the $2$-sphere  having at least one singular fiber. It is known that the minimal number of singular fibers in all torus Lefschetz fibrations is $12$. One can conclude that $N_1=12$ by the existence of torus Lefschetz fibrations with $12$ singular fibers on the elliptic surface $E(1)=\mathbb{C} P^{2}\#9\overline{\mathbb{C} P^{2}}$. A genus-$2$ Lefschetz fibration with $14$ singular fibers on a symplectic $4$-manifold which is an exotic copy of $\mathbb{C} P^{2}\#7\overline{\mathbb{C} P^{2}}$ was constructed by Baykur and Korkmaz~\cite{bki}. By~\cite[Theorem $2$]{bki}, one can obtain that $N_2=14$.

Let us define $M_g$ as the minimal number of singular fibers in all hyperelliptic Lefschetz fibrations on a simply-connected $4$-manifolds over $\mathbb{S}^{2}$ having at least one singular fiber. By the same argument above, $M_1=12$ and $M_2=14$.

The purpose of the present paper is to estimate the numbers of $N_g$ and $M_g$. In this direction, in Section~\ref{S2}, we first give some preliminary information and results. In Section~\ref{S3}, we explicitly construct two genus-$4$ Lefscherz fibrations on simply-connected $4$-manifolds. To do this we first mention the positive factorizaton $W$ for a genus-$3$ hyperelliptic Lefschetz fibration given by Baykur~\cite{b1}. Then we construct a genus-$4$ nonhyperelliptic Lefschetz fibration $(X_1,f_1)$ using the Baykur's monodromy and the monodromy that gives the smallest genus-$2$ Lefschetz fibration. Here, we use the breeding technique to construct this positive factorization (see~\cite{bk1,b1,bh} for more applications of this technique). We also prove that the $4$-manifold $X_1$ is an exotic copy of $\mathbb{C} P^{2}\#8\overline{\mathbb{C} P^{2}}$ (Theorem~\ref{thmex4}). Similarly, we produce another monodromy which gives a genus-$4$ hyperelliptic Lefschetz fibration $(X_2,f_2)$ using the monodromy of generalized Matsumoto's fibration for $g=4$ and again the monodromy of the smallest genus-$2$ Lefschetz fibration. We prove that the $4$-manifold $X_2$ is an exotic $\mathbb{C} P^{2}\#9\overline{\mathbb{C} P^{2}}$ (Theorem~\ref{thmexh4}). In the last section, we examine the numbers $N_g$ and $M_g$ for $g\leq 4$. We first give a different proof for the result $N_2=M_2=14$ given in~\cite[Theorem $2$]{bki}  (Lemma~\ref{lemn2}). Moreover, we prove that $M_3=18$ (Lemma~\ref{lemn3}). Similarly, using the existence of genus-$4$ nonhyperelliptic and hyperelliptic Lefschetz fibrations $(X_1,f_1)$ and $(X_2,f_2)$ constructed in Section~\ref{S3}, we conclude that $N_4\leq 23$ and $M_4\leq 24$. We then give  better estimates for the number $M_4$ (Lemma~\ref{lemm4}). Finally, we discuss the numbers $N_g$ and $M_g$ for higher genus.

\section{Preliminaries}
\label{S2}
This section presents the necessary background and the known results used
in our proofs.
\subsection{Mapping Class Groups}
\  Let us denote a compact connected oriented smooth surface of genus $g$
with $n\geq 0$ boundary components by $\Sigma_{g}^{n}$. Let $\Mod_{g}^{n}$ denote the mapping class group of $\Sigma_{g}^{n}$, i.e., the group of isotopy classes of orientation-preserving self-diffeomorphisms of $\Sigma_{g}^{n}$ fixing all points on the
boundary. We assume all isotopies are identity on the boundary. When $n=0$, we will denote $\Mod_{g}^{n}$ and $\Sigma_{g}^{n}$ by $\Mod_g$ and $\Sigma_{g}$, respectively. Throughout the paper we do not distinguish a 
 diffeomorphism from its isotopy class. For the composition of two diffeomorphisms, we
use the functional notation; if $g$ and $h$ are two diffeomorphisms, then
the composition $gh$ means that we apply $h$  first and then $g$.
\indent

Now, let us remind the following basic properties of Dehn twists. Let $a$ and $b$ be 
simple closed curves on $\Sigma_{g}^{n}$ and $f\in \Mod_g^{n}$.
\begin{itemize}
\item \textbf{Commutativity:} If $a$ and $b$ are disjoint, then $t_at_b=t_bt_a$.
\item \textbf{Conjugation:} If $f(a)=b$, then $ft_af^{-1}=t_b$.
\end{itemize}
( Here, $t_a$ denotes the positive Dehn twist about a simple closed curve $a$.)

\subsection{Lefschetz Fibrations} We remind some basic definitions and facts about Lefschetz fibrations. Throughout the paper we denote the $2$-sphere by $\mathbb{S}^{2}$. Let $M$ be a closed connected oriented smooth $4$-manifold. A \textit{Lefschetz fibration} on $M$  is a smooth surjective map if it has only finitely many critical points $\lbrace p_1,p_2,\ldots,p_n \rbrace$ such that around each of which it is expressed in the form of $f(z_1,z_2)=z_1^{2}+z_2^{2}$ with respect to some local complex coordinates compatible with the orientations of $M$ and $\mathbb{S}^{2}$. (In general, the base of a Lefschetz fibration
can be a closed orientable surface of genus $h\geq 0$, but throughout this paper, we only consider $\mathbb{S}^{2}$.) The genus-$g$ of a regular fiber is defined to be \textit{the genus of the fibration}. The inverse image of a critical value is called a \textit{singular fiber}. We assume that each singular fiber  contains only one critical point, which can be obtained by a small perturbation. Each singular fiber is obtained by collapsing a simple closed curve, called \textit{vanishing cycle}, on a nearby regular fiber to a point. If the vanishing cycle is non-separating (respectively separating), then the corresponding singular fiber is called \textit{irreducible} (respectively \textit{reducible}). Throughout the paper, we also assume that all Lefschetz fibrations are nontrivial and relatively minimal, i.e. it has at least one singular fiber and no fiber
containing a $(-1)$-sphere.

A Lefschetz fibration can be described via its monodromy, which is an element in the mapping class group $\Mod_g$. The monodromy of a Lefschetz fibration $f: M \to \mathbb{S}^{2}$ is given by a positive factorization 
\[
t_{a_1}t_{a_2} \cdots t_{a_m}=1
\]
in $\Mod_g$ up to \textit{Hurwitz moves} (exchanging subwords $t_{a_i}t_{a_{i+1}}=t_{a_{i+1}}t_{t_{a_{i+1}}^{-1}(a_i)}$) and \textit{global conjugations} (changing each $t_{a_i}$ with $t_{\phi(a_i)}$ for some $\phi \in \Mod_g$), where $a_i$'s are vanishing cycles of the singular fibers. A map $\sigma: M \to \mathbb{S}^{2}$ is a \textit{section} if $f\circ\sigma=id_{\mathbb{S}^{2}}$. If a positive relation $t_{a_1}t_{a_2} \cdots t_{a_m}=1$ in $\Mod_g$ has a lifting to $\Mod_g^{k}$ so that 
\[
t_{\tilde{a}_1}t_{\tilde{a}_2} \cdots t_{\tilde{a}_m}=t_{\delta_1}^{n_1}t_{\delta_2}^{n_2}\cdots t_{\delta_k}^{n_k},
\]
where each $n_i$ is integer and $\delta_i$ is a boundary curve, then the Lefschetz fibration $f : M \to \mathbb{S}^{2}$ admits $k$ disjoint sections $S_1,\ldots,S_k$, where $S_j$ is of self-intersection $-n_{j}$ and vice versa \cite{bkm}. We say that two Lefschetz fibrations $f_1: M_1 \to \mathbb{S}^{2}$ and $f_2: M_2 \to \mathbb{S}^{2}$ are \textit{isomorphic} if there exist orientation preserving diffeomorphisms $G:M_1 \to M_2$ and $g: \mathbb{S}^{2} \to \mathbb{S}^{2}$ such that $f_2\circ G=g\circ f_1$.

The hyperelliptic mapping class group $\HMod_g$ of $\Sigma_g$ is defined as the subgroup of $\Mod_g$ that is the centralizer of a hyperelliptic involution $\iota: \Sigma_g \to \Sigma_g$. A Lefschetz fibration is said to be \textit{hyperelliptic} if its vanishing cycles are invariant under the hyperelliptic involution $\iota$ up to isotopy. 

For a genus-$g$ Lefschetz fibration $f:M\to \mathbb{S}^{2}$, the Euler characteristic, $e(M)$, of the $4$-manifold $M$ can be computed as 
\[
e(M)=4-4g+n+s,
\]
where $n$ and $s$ are the numbers of nonseparating and separating vanishing cycles, respectively. Also we define the following invariant associated to the $4$-manifold $M$:
\[
\chi_{h}(M)=\dfrac{e(M)+\sigma(M)}{4},
\]
where $\sigma(M)$ is the signature of $M$. Let us note that if $M$ is a complex surface, $\chi_{h}(M)$ is the holomorphic Euler characteristic.

It follows  from the theory of Lefschetz fibrations that if a Lefschetz fibration $f: M \to \mathbb{S}^{2}$ with a regular fiber $\Sigma_g$ and the monodromy $t_{\alpha_1}t_{\alpha_2} \cdots t_{\alpha_m}=1$ admits a section, then the fundamental group $\pi_{1}(M)$ of $M$ is isomorphic to the group $\pi_{1}(\Sigma_g)$ divided by the normal closure of the vanishing cycles (cf~\cite{gs}), that is,
\[
\pi_{1}(M)\cong \pi_1(\Sigma_g) / \langle \alpha_1, \alpha_2,\ldots \alpha_m\rangle
\]

The signature $\sigma(M)$ of $M$, which is another invariant of the Lefschetz fibration $f:M \to \mathbb{S}^{2}$ can be computed using several techniques. For instance, Endo and Nagami~\cite{e11}
gave a useful method which uses the signatures of the relations involved in its monodromy. For an integer-valued function $I_g$ on the set of relators of $\Mod_g$ (see \cite{e11} for its definition and properties), the following  theorem holds:
\begin{theorem}\cite{e11}\label{signen}
Let $f:M \to \mathbb{S}^{2}$ be a genus-$g$ Lefschetz fibration with the monodromy $t_{c_1}t_{c_2}\cdots t_{c_n}=1$. Then the signature of $M$ is
\[
\sigma(M)=I_g(c_1c_2\cdots c_n).
\]
\end{theorem}
\noindent
This method allows us to compute the signature of a Lefschetz fibration $f:M \to \mathbb{S}^{2}$ as the sum of basic relations in its monodromy.
Let us recall some signatures that we will need later. For the proof, see \cite{e11}.

\begin{itemize}
\item  $I_{g}(a)=-1$, where $a$ is the isotopy class of a separating curve.
\item  $ I_{g}((B_{0}B_{1}\cdots B_{g}C)^{2})=-4$ if $g$ is even.
\end{itemize}
 Here, the word $(B_{0}B_{1}\cdots B_{g}C)^{2}$ is the relator coming from Matsumoto's relation which is explained later (see~\ref{mats2}).

One can also use  Ozbagci's algorithm~\cite{oz} to compute the signature of a Lefschetz fibration $f:M \to \mathbb{S}^{2}$. For hyperelliptic Lefschetz fibrations, we have the following  useful lemma:

\begin{lemma}\cite{eh,mat1,mat}\label{lem2}
Let $f\colon M \to \mathbb{S}^{2}$ be a genus-$g$ hyperelliptic Lefschetz fibration. 
Let $n$ and $s=\sum_{h=1}^{[g/2]}s_{h}$ denote the numbers of non-separating
 and separating vanishing cycles of this fibration, respectively, where $s_h$ is the number of separating vanishing cycles that separate the genus-$g$ surface into two surfaces one of which has genus $h$. Then the signature of $M$ is
\begin{eqnarray*}
\sigma(M)&=&-\dfrac{g+1}{2g+1}n+\sum_{h=1}^{[g/2]} \bigg(\dfrac{4h(g-h)}{2g+1}-1\bigg)s_{h}.
\end{eqnarray*}
\end{lemma}

\begin{remark}
 One can easily obtain that $\sigma(M)\leq n-s-2$ using $b_1(M)\leq 2g-1$ by the handlebody decomposition of nontrivial Lefschetz fibrations $f: M \to \mathbb{S}^{2}$ and the fact that such fibrations have at least one non-separating vanishing cycle. If the Lefschetz fibration $f: M \to \mathbb{S}^{2}$ is hyperelliptic, then Ozbagci~\cite{oz} proved that $\sigma(M)\leq n-s-4$. Later, for every Lefschetz fibration $f: M \to \mathbb{S}^{2}$, Cadavid~\cite{Cadavid} improved the upper bound of signature $\sigma(M)$ showing that
 \begin{eqnarray}
\sigma(M)\leq n-s-2(2g-b_1(M)).
\end{eqnarray}
When the $4$-manifold $X$ is simply-connected, the above inequality turns out to be
 \begin{eqnarray}\label{eq2} 
\sigma(X)\leq n-s-4g.
\end{eqnarray}
\end{remark}

For a closed orientable surface of genus $g\geq1$, the first homology group $H_1(\HMod_g;\mathbb{Z})$ of the hyperelliptic mapping class
  group $\HMod_{g}$ has the following isomorphism:
\[ H_1(\HMod_{g};\mathbb{Z})\cong
  \begin{cases}
   \mathbb{Z}/ 4(2g+1),&\text{if $g$ is odd},\\
   \mathbb{Z}/ 2(2g+1),&\text{if $g$ is even},\\
    \end{cases}
\]
which can be proven by the presentation of the hyperelliptic mapping class
  group $\HMod_{g}$~\cite{bh}. In the hyperelliptic mapping class
  group $\HMod_{g}$ all Dehn twists about non-separating simple closed curves are nontrivial and each of which maps to the same generator in $H_1(\HMod_g;\mathbb{Z})$ under the natural map $\HMod_{g} \to H_1(\HMod_g;\mathbb{Z})$. Thus, the number of twists of a factorization of identity in $\HMod_{g}$ consisting of positive Dehn twists about non-separating simple closed curves is divisible by  $2(2g+1)$ (respectively $4(2g+1)$) if $g$ is even (respectively odd). We say that a separating simple closed curve on $\Sigma_g$ is of \textit{type $h$} if it separates $\Sigma_g$ into two subsurfaces of genera $h$
 and $g-h$. It is known that each  separating simple closed curve of type $h$ can be written as a product of $2h(4h+2)$ positive Dehn twists about non-separating simple closed curves. Therefore, we have the following lemma which gives a relation between the number of non-separating vanishing cycles and that of separating vanishing cycles in a genus-$g$ hyperelliptic Lefschetz fibration:

\begin{lemma}\label{lem1}
Let $n$ (or $s$) be the number of non-separating (resp. separating) vanishing cycles in a genus-$g$ hyperelliptic Lefschetz fibration over $\mathbb{S}^{2}$. Then, we have
\begin{eqnarray}\label{eq1} 
n+\sum_{h=1}^{[g/2]}2h(4h+2)s_{h}\equiv  \begin{cases} 0 \pmod{4(2g+1)},&\text{if $g$ is odd,} \\ 
0 \pmod{2(2g+1)},&\text{if $g$ is even},  \end{cases}
\nonumber 
\end{eqnarray}
where $s_h$ is the number of separating vanishing cycles of type $h$ with $s=\sum_{h=1}^{[g/2]}s_{h}$.
\end{lemma}

  \par
  \subsection{The smallest genus two Lefschetz fibration.}
  
\begin{figure}[hbt!]
\begin{center}
\scalebox{0.3}
{\includegraphics{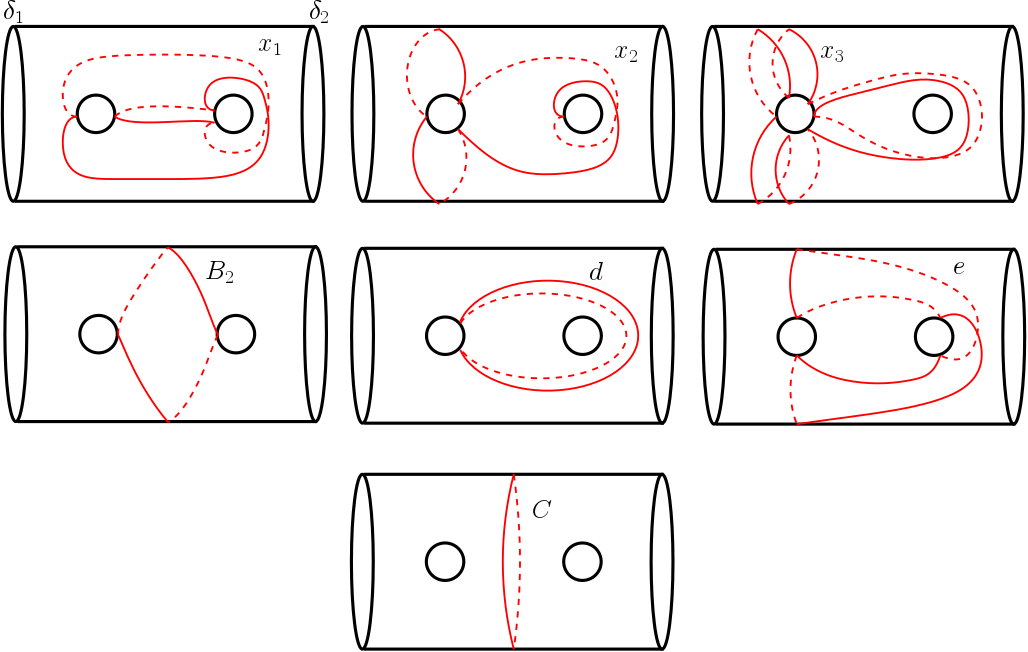}}
\caption{The curves $x_{i}$, $B_2$, $C$, $d$ and $e$ on the surface $\Sigma_{2}^{2}$.}\label{BK}
\end{center}
\end{figure}
 In~\cite{bki}, Baykur and Korkmaz obtained the following relation in $\Mod_{2}^{1}$:
  \[
t_{e}t_{x_{1}}t_{x_{2}}t_{x_{3}} t_dt_Ct_{x_{4}}=t_\delta,
\]
which can be rewritten as 
\[
t_{e}t_{x_{1}}t_{x_{2}}t_{x_{3}} t_dt_{B_{2}}t_C=t_\delta,
\]
where $t_c(x_4)=B_2$. Set $T=t_{e}t_{x_{1}}t_{x_{2}}t_{x_{3}} t_dt_{B_{2}}t_C$. They also showed that the positive factorization $T=t_\delta$ realizes the smallest genus-$2$ Lefschetz fibrations whose total space is diffeomorphic to $(\mathbb{T}^{2}\times\mathbb{S}^2)\#3\overline{\mathbb{C} P^{2}}$. Stipsicz and Yun~\cite{StY} obtained the following lifting of $T$:
\begin{eqnarray}\label{eqsy}
T=t_{e}t_{x_{1}}t_{x_{2}}t_{x_{3}} t_dt_{B_{2}}t_C=t_{\delta_1}t_{\delta_2}
\end{eqnarray}
 in $\Mod_{2}^{2}$, where the curves $x_i$, $B_2$, $C$, $d$, $e$, $\delta_1$ and $\delta_2$  are as depicted in Figure~\ref{BK}.
 \subsection{Generalized Matsumoto's relation.}\label{mats2}
A relation with eight positive Dehn twists was discovered by Matsumoto~\cite{mat}, which is the global monodromy of a Lefschetz fibration on $(\mathbb{T}^{2}\times\mathbb{S}^2)\#4\overline{\mathbb{C} P^{2}}$. It is later generalized to higher genus surfaces by Korkmaz~\cite{k1}, independently by Cadavid~\cite{Cadavid} and recently by a different proof~\cite{dmp}. A lift of this relation to $\Mod_{g}^{1}$ was first discovered by Ozbagci
and Stipsicz~\cite{os} and another lift to $\Mod_{g}^{2}$ by Korkmaz~\cite{kl}. However, we will use the following lift to $\Mod_{g}^{2}$ proved by Hamada~\cite{h}:
\[
V_g=\left\{\begin{array}{lll}
(t_{B_0}t_{B_1}\cdots t_{B_g}t_{C})^{2}& \textrm{if} & g=2k,\\
(t_{B_0}t_{B_1}\cdots t_{B_g}t_{a}^{2}t_{b}^{2})^{2}& \textrm{if} & g=2k+1.\\
\end{array}\right.
\]
where $\delta_i$ are the boundary parallel curves, and the curves $B_i$ and $C$ are as shown in Figure~\ref{ML}.
\begin{figure}[hbt!]
\begin{center}
\scalebox{0.35}
{\includegraphics{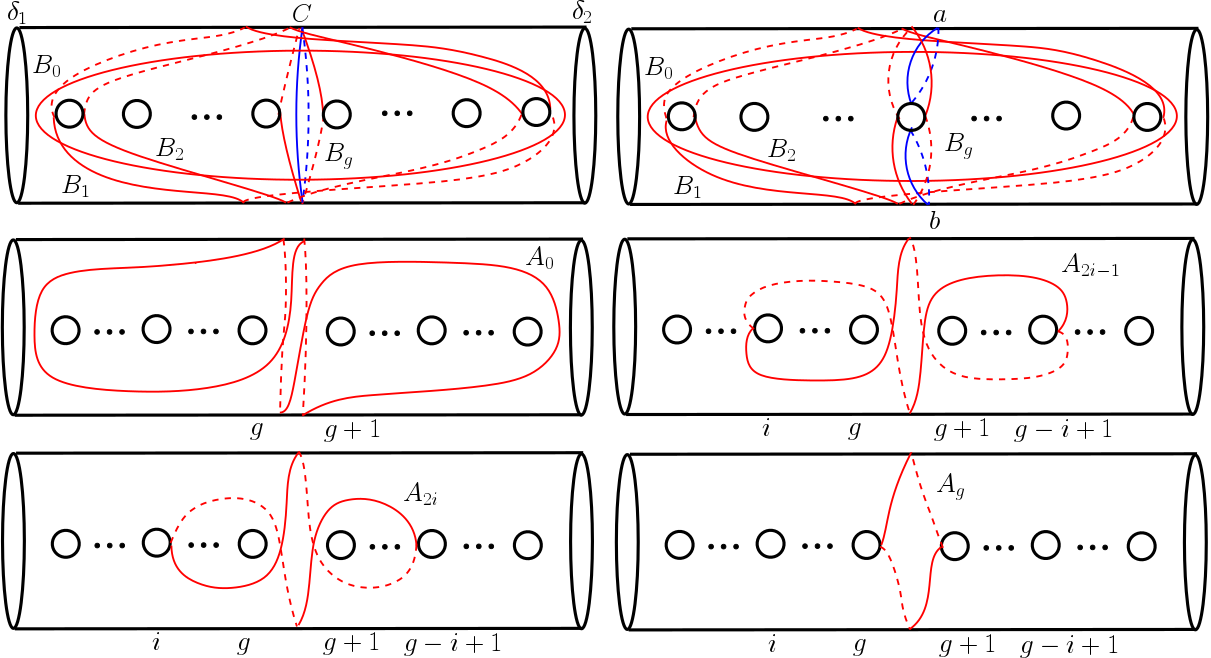}}
\caption{The curves $B_{i}$, $A_i$, $C$, $a$ and $b$ on the surface $\Sigma_{g}^{2}$.}\label{ML}
\end{center}
\end{figure}

 One may rewrite the generalized Matsumoto's relation for $g=2k$ as follows:
  \begin{eqnarray*}
(t_{B_0} t_{B_1} \cdots t_{B_g}t_{C})^{2}&=&(t_{B_0} t_{B_1} \cdots t_{B_g} t_{C})(t_{B_0} t_{B_1} \cdots t_{B_g} t_{C})\\
 &=&(t_{C}t_{t_{C}^{-1}(B_0)}t_{t_{C}^{-1}(B_1)}\cdots t_{t_{C}^{-1}(B_g)})(t_{B_0} t_{B_1} \cdots t_{B_g} t_{C})=t_{\delta_1}t_{\delta_2}.
  \end{eqnarray*}
  Hence, it follows from the Dehn twist $t_C$ commutes with $t_{\delta_1}$ and $t_{\delta_2}$ that we get
 \begin{eqnarray}\label{mat2}
 V_g= t_{C}^{2} t_{A_0} t_{A_1} \cdots t_{A_g}t_{B_0} t_{B_1} \cdots t_{B_g}=t_{\delta_1}t_{\delta_2}
 \end{eqnarray}
in $\Mod_{g}^{2}$, where each $A_i=t_{t_{C}^{-1}(B_i)}$ is shown in Figure~\ref{ML}. Note that the total space of the genus-$g$ Lefschetz fibration is diffeomorphic to $(\Sigma_k \times\mathbb{S}^2)\#4\overline{\mathbb{C} P^{2}}$ (respectively $(\Sigma_k\times\mathbb{S}^2)\#8\overline{\mathbb{C} P^{2}}$) if $g=2k$ (respectively $g=2k+1$).
  \section{Small Lefschetz fibrations of fiber genus $4$ on simply connected $4$-manifolds}\label{S3}
In this section our aim is to construct small genus-$4$ Lefschetz fibrations on  simply connected $4$-manifolds. In order to make our construction, first we will derive a positive factorization of $t_{\delta_1}t_{\delta_2}$ in $\Mod_{2}^{2}$ with $(n,s)=(4,3)$, which will be one of our building blocks. For completeness of our contructions, using the breeding technique~\cite{bk1,b1,bh}, we also give the positive factorization $W$ of $t_{\delta_1}t_{\delta_2}$ in $\Mod_{3}^{2}$ with $(n,s)=(12,6)$ constructed by Baykur~\cite{b1}. Afterwards, we derive a nonhyperelliptic genus-$4$ Lefschetz fibration on an exotic copy of $\mathbb{C} P^{2}\#8\overline{\mathbb{C} P^{2}}$ by breeding the factorization $W$ with the Matsumoto's genus-$2$ factorization $V_2$ given in (\ref{mat2}). Moreover, we construct a hyperelliptic genus-$4$ Lefschetz fibration on an exotic copy of $\mathbb{C} P^{2}\#9\overline{\mathbb{C} P^{2}}$ using again the breeding technique to the generalized Matsumoto's factorization $V_4$ given in (\ref{mat2}) and the factorizations which give the smallest genus-$2$ Lefschetz fibration.
\begin{figure}[h]
\begin{center}
\scalebox{0.3}
{\includegraphics{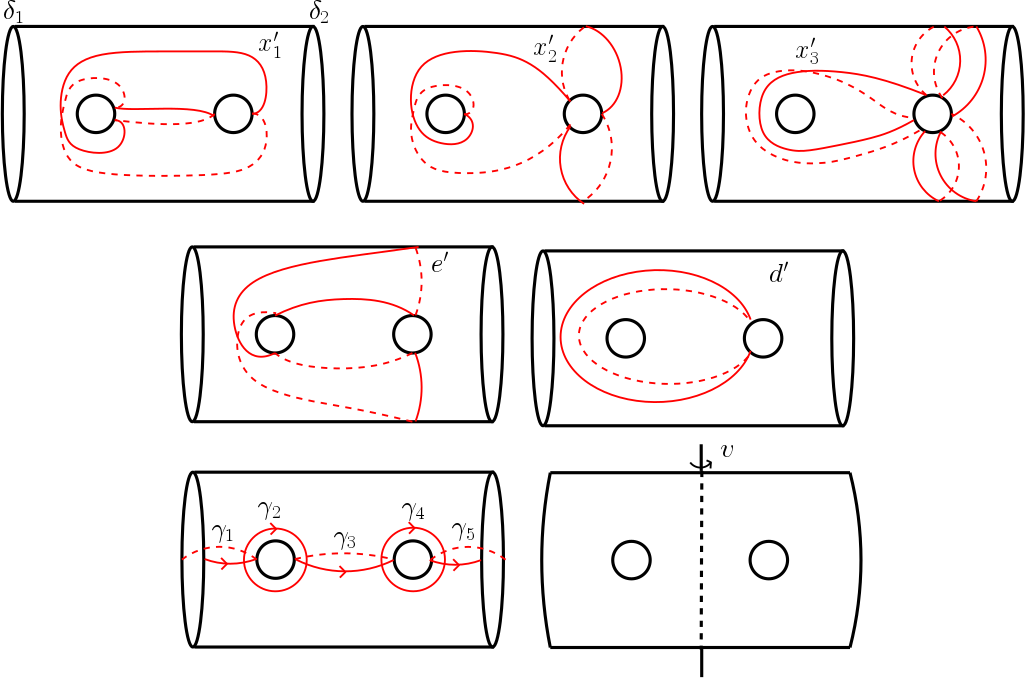}}
\caption{The curves  $x_i'$, $d'$, $e'$, $\gamma_j$'s on the surface $\Sigma_{2}^{2}$ and the involution $\upsilon$ on the surface $\Sigma_{2}$. }\label{TC}
\end{center}
\end{figure}

\begin{figure}
\begin{center}
\scalebox{0.28}{\includegraphics{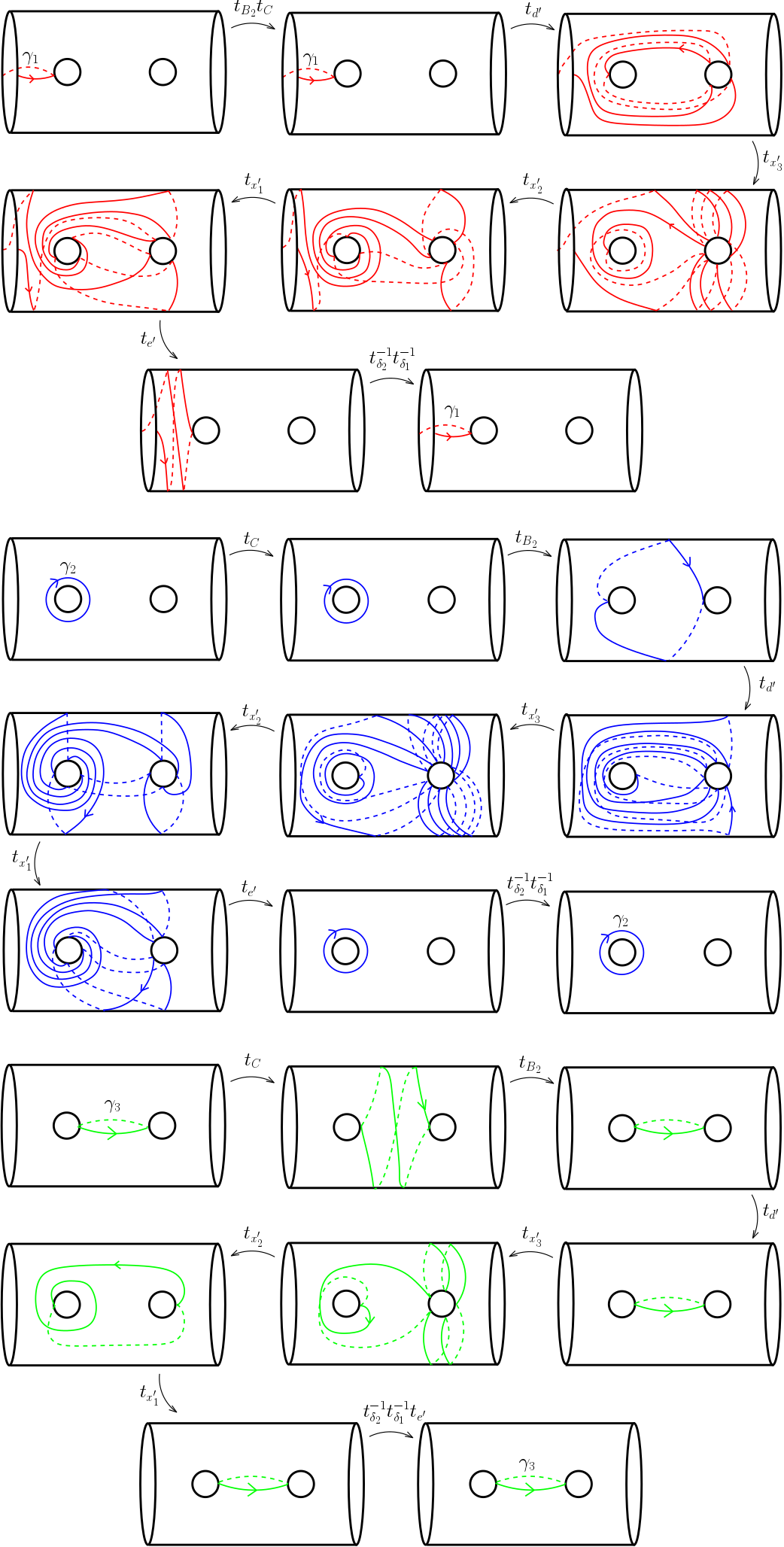}}
\caption{$t_{\delta_2}^{-1}t_{\delta_1}^{-1}t_{e'}t_{x_1'}t_{x_2'}t_{x_3'}t_{d'}t_{B_2}t_{C}(\gamma_i)=\gamma_i$ for $i=1,2,3.$}\label{G1-2-3}
\end{center}
\end{figure}

\begin{figure}
\begin{center}
\scalebox{0.28}{\includegraphics{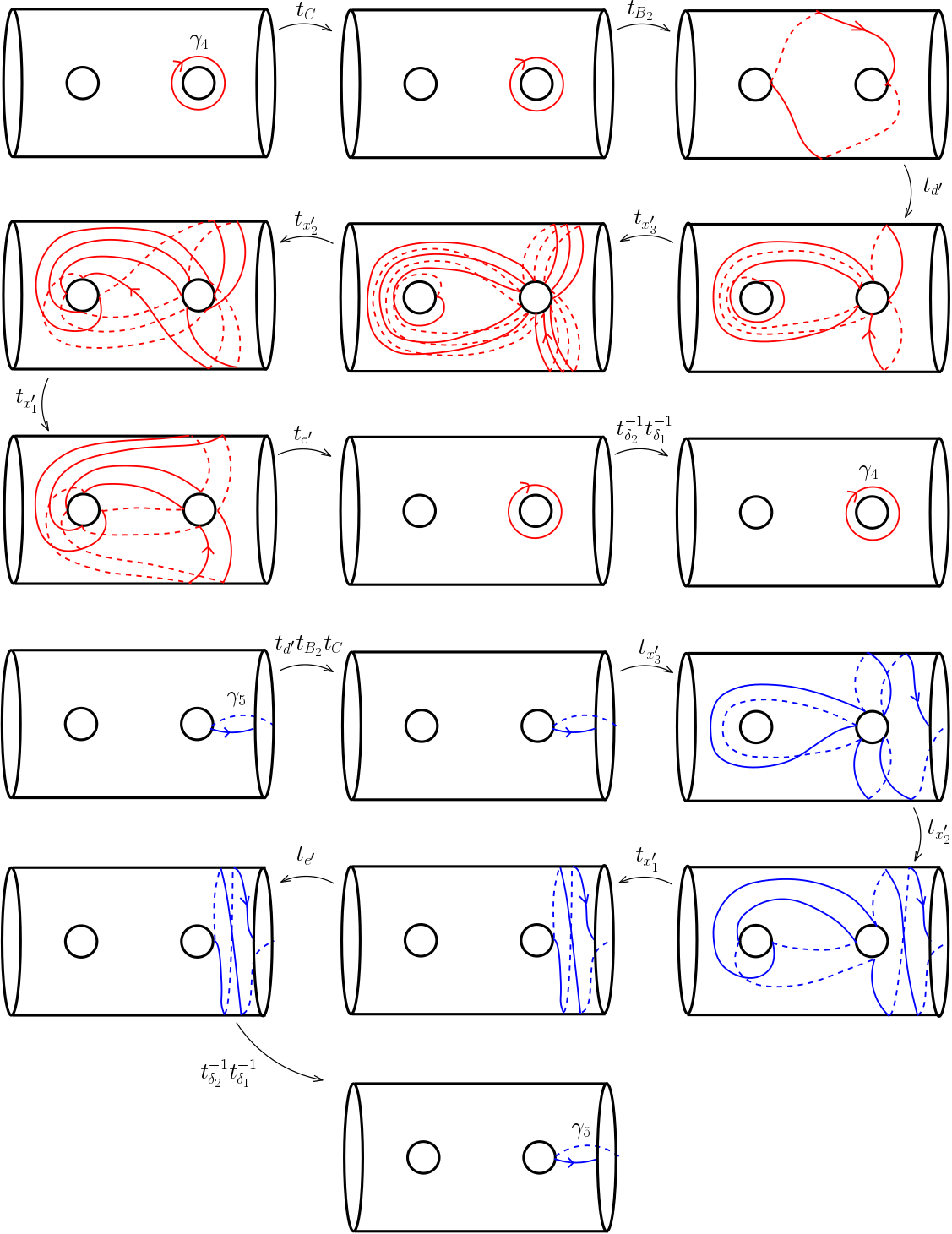}}
\caption{$t_{\delta_2}^{-1}t_{\delta_1}^{-1}t_{e'}t_{x_1'}t_{x_2'}t_{x_3'}t_{d'}t_{B_2}t_{C}(\gamma_i)=\gamma_i$ for $i=4,5$}\label{G4-5}
\end{center}
\end{figure}

Consider the curves $x_1'$, $x_2'$, $x_3'$, $B_2$, $C$, $d'$ and $e'$ on the genus-$2$ surface $\Sigma_{2}^{2}$ shown in Figure~\ref{TC}. Note that the closed surface $\Sigma_2$ is embedded in $\mathbb{R}^3$ in such a way that it is invariant under the involution $\upsilon$ shown in Figure~\ref{TC}. One can observe that these curves can be obtained by the applying the involution $\upsilon$ to the curves contained in factorization (\ref{eqsy}). 

We now prove a factorization of  $t_{\delta_1}t_{\delta_2}$ in $\Mod_{2}^{2}$.

\begin{lemma} There is a relation 
\begin{eqnarray}\label{eq31}
t_{e'}t_{x_1'}t_{x_2'}t_{x_3'}t_{d'}t_{B_2}t_{C}=t_{\delta_1}t_{\delta_2}.
\end{eqnarray}
in $\Mod_{2}^2$, where the curves $x_i'$, $B_2$, $C$, $d'$ and $e'$ are as in Figure~\ref{TC}.
\end{lemma}
\begin{proof}
Since the collection of simple closed curves and simple proper arcs $\lbrace\gamma_i \rbrace$ fills the genus-$2$ surface $\Sigma_{2}^{2}$ as shown in Figure~\ref{TC}, we will prove the relation (\ref{eq31}) by showing that the oriented three curves $\gamma_2$, $\gamma_3$, $\gamma_4$ and the two arcs $\gamma_1$ and $\gamma_5$ are fixed (up to isotopy) under the map
\[
t_{\delta_2}^{-1}t_{\delta_1}^{-1}t_{e'}t_{x_1'}t_{x_2'}t_{x_3'}t_{d'}t_{B_2}t_{C}.
\]
Indeed, Figures~\ref{G1-2-3} and \ref{G4-5} show that the collection $\lbrace \gamma_i \rbrace$ are fixed and also their given orientations are preserved under this map. This finishes the proof.
\end{proof}

We remark that any genus-$2$ Lefschetz fibration prescribed by a monodromy with $(n,s)=(4,3)$ has total space diffeomorphic to $(\mathbb{T}^{2}\times\mathbb{S}^2)\#3\overline{\mathbb{C} P^{2}}$. So the corresponding genus-$2$ Lefschetz fibration to our monodromy is diffeomorphic to the smallest genus-$2$ Lefschetz fibration constructed by Baykur and Korkmaz~\cite{bki}. 


One of the building block of our monodromy construction is the monodromy of a genus-$3$ Lefschetz fibration on an exotic copy of $\mathbb{C} P^{2}\#7\overline{\mathbb{C} P^{2}}$ given by Baykur~\cite{b1}. For completeness, we give the construction of this monodromy in detail.

Set $T'=t_{e'}t_{x_1'}t_{x_2'}t_{x_3'}t_{d'}t_{B_2}t_{C}$ or $T'=t_{e'}P't_{C}$ so that $T'=t_{e'}P't_{C}=t_{\delta_1}t_{\delta_2}$ in $\Mod_{2}^{2}$. Since $t_{C}$ and the factorization $t_{e'}P'$ commute, we have
\begin{eqnarray}\label{eqwtilde}
T'=t_{e'}P't_{C}=t_{C}t_{e'}P'=t_{\delta_1}t_{\delta_2}
\end{eqnarray}

Similarly, let us write the positive factorization $T$ in (\ref{eqsy}) as $T=t_{e}Pt_C$, where $P=t_{x_{1}}t_{x_{2}}t_{x_{3}} t_dt_{B_{2}}$ so that $T=t_{e}Pt_C=t_{\delta_1}t_{\delta_2}$ in $\Mod_{2}^{2}$. By the commutativity of $t_C$ and $t_{e}P$, we have
 \begin{eqnarray}\label{eqw}
 T=t_{e}Pt_C=t_Ct_{e}P=t_{\delta_1}t_{\delta_2}.
 \end{eqnarray}
 \begin{figure}[h]
\begin{center}
\scalebox{0.5}{\includegraphics{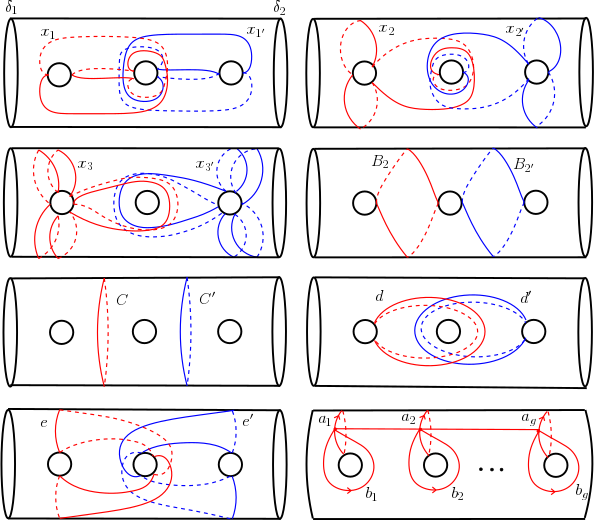}}
\caption{The curves coming from the factorizations $T$ and $T'$ and the generators of $\pi_{1}(\Sigma_g)$}\label{TA3}
\end{center}
\end{figure} 
 Let us embed the relation (\ref{eqw}) into $\Mod_{3}^{2}$ so that the boundary parallel curve $\delta_2$ in Figure~\ref{BK} is mapped to the curve $C'$ in Figure~\ref{TA3}. Hence we get the following relation:
 \begin{eqnarray}\label{eq20}
 T=t_{e}Pt_C=t_Ct_{e}P=t_{\delta_1}t_{C'}.
 \end{eqnarray}
 We also embed the relation (\ref{eqwtilde}) into $\Mod_{3}^{2}$ so that the curve $\delta_1$ in Figure~\ref{BK} is mapped to the curve $C$ in Figure~\ref{TA3}. In this case, the curves $B_2$ and $C$ appearing in the facrorization $T'$ are mapped to the curves $B_2'$ and $C'$ in Figure~\ref{TA3} so that the factorization $P'=t_{x_1'}t_{x_2'}t_{x_3'}t_{d'}t_{B_2'}$. We thus have 
\begin{eqnarray}\label{eq21}
T'=t_{e'}P't_{C'}=t_{C'}t_{e'}P'=t_{C}t_{\delta_2}.
 \end{eqnarray}

 Combining the relations (\ref{eq20}) and (\ref{eq21}), the following relation in $\Mod_{3}^{2}$ holds:
\[
TT'=(t_{C}t_eP)( t_{e'}P't_{C'})=t_{\delta_1}t_{C'}t_{C}t_{\delta_2}.
\]
By the fact that $t_C$ commutes with $t_ePt_{e'}P't_{C'}$ and the curves $\delta_1$, $\delta_2$, $C$ and $C'$ are all disjoint, the relation can be written as 
\[
TT'=t_ePt_{e'}P't_{C'}t_C=t_{\delta_1}t_{\delta_2}t_{C'}t_C,
\]
which gives the following relation
\[
t_ePt_{e'}P't_{C'}t_Ct_{C}^{-1}t_{C'}^{-1}=t_{\delta_1}t_{\delta_2}.
\]
Finally we obtain the following identity in $\Mod_{3}^{2}$:
\begin{eqnarray}\label{eqf}
t_ePt_{e'}P'=t_{\delta_1}t_{\delta_2}.
\end{eqnarray}
\begin{figure}[h]
\begin{center}
\scalebox{0.3}{\includegraphics{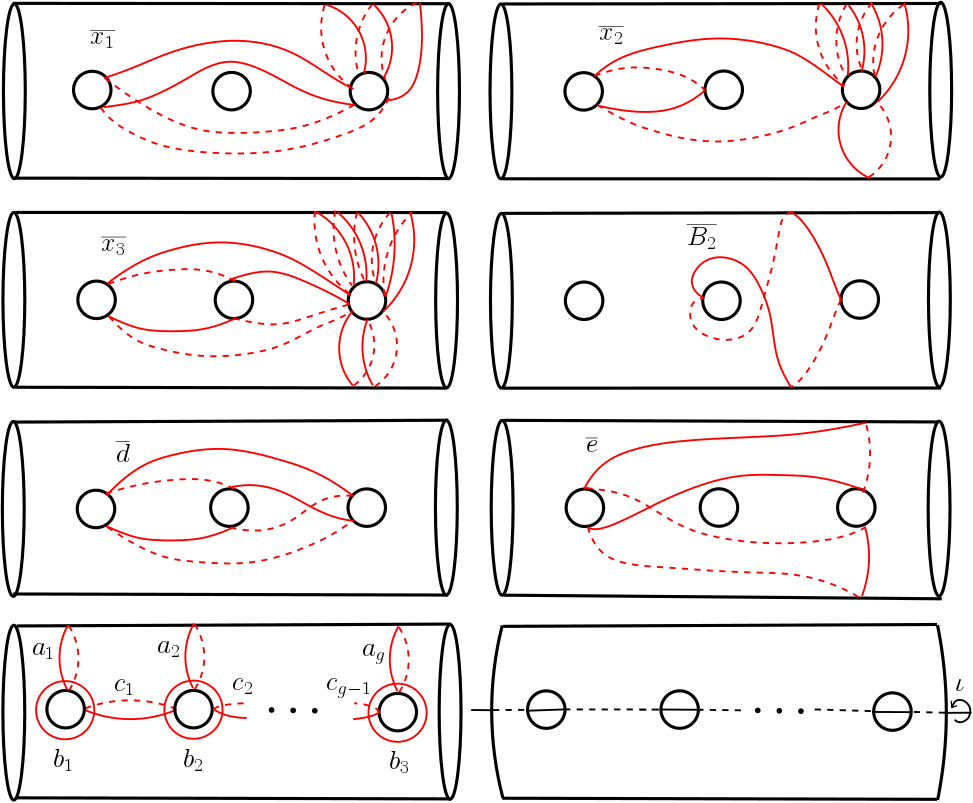}}
\caption{The curves $\overline{x_i}'s$, $\overline{B_2}$, $\overline{d}$, $\overline{e}$ and the hyperelliptic involution $\iota$}\label{TAB3}
\end{center}
\end{figure}

Consider the diffeomorphism $\varphi=t_{a_3}^{3}t_{b_2}t_{c_{1}}$, where the curves  are shown in Figure~\ref{TAB3}. One can easily verify that $\varphi(C)=e$. Then by conjugation of $T'$ by $\varphi$ we obtain the following factorization of $t_{\delta_1}t_{\delta_2}$ in $\Mod_{3}^{2}$:
\[
(T')^{\varphi}=t_{\varphi(C')}t_{ \varphi(e')}(P'^{\varphi})=t_{\varphi(C)}t_{\varphi(\delta_2)},
\]
which can be written as 
\[
(T')^{\varphi}=t_{C'}t_{\varphi(e')}t_{\varphi(x_{1}')}t_{\varphi(x_{2}')}
t_{\varphi(x_{3}')}t_{\varphi(d')}
t_{\varphi(B_{2}')}=t_et_{\delta_2},
\]
since $\varphi(C')=C'$ and $\varphi(\delta_2)=\delta_2$, where $(T')^{\varphi}$ denotes the conjugate factorization. Let us denote $(P')^{\varphi}$ by $\overline{P}$ and denote $\varphi(\alpha)=\overline{\alpha}$ for every curve $\alpha$ appearing in the factorization $t_{e'}P'$. Therefore, we have the following relation:
\begin{eqnarray}\label{eq3rd}
(T')^{\varphi}=t_{C'}t_{\overline{e}}\overline{P}=t_{C'}t_{\overline{e}}t_{\overline{x_1}}t_{\overline{x_2}}t_{\overline{x_3}}t_{\overline{d}}t_{\overline{B_2}}=t_et_{\delta_2},
\end{eqnarray}
where the curves $\overline{e}$,  $\overline{x_i}$'s,  $\overline{d}$ and  $\overline{B_2}$ are shown in Figure~\ref{TAB3}. Therefore, the relation  (\ref{eq3rd}) together with the relation (\ref{eqf}) give rise to the following relation in $\Mod_{3}^{2}$:
\[
(t_ePt_{e'}P')(t_{C'}t_{\overline{e}}\overline{P})=(t_{\delta_1}t_{\delta_2})(t_et_{\delta_2}),
\]
which implies that
\[
t_{e}^{-1}t_ePt_{e'}P't_{C'}t_{\overline{e}}\overline{P}=t_{\delta_1}t_{\delta_2}^{2}
\]
by commutativity of the curve $t_e$ and $t_{\delta_1}t_{\delta_2}$. Finally, by canceling the $t_e$ factors we get the desired equation in $\Mod_{3}^{2}$:
\begin{eqnarray*}
Pt_{e'}P't_{C'}t_{\overline{e}}\overline{P}=t_{\delta_1}t_{\delta_2}^{2}.
\end{eqnarray*}
This relation is the monodromy factorization for our genus-$3$ Lefschetz fibration. Observe that it admits two sections: one is of $(-1)$ self-intersection and the other is of $(-2)$. Capping off the boundary components, the factorization $Pt_{e'}P't_{C'}t_{\overline{e}}\overline{P}$ gives the following factorization of the identity in $\Mod_3$:
\begin{eqnarray}\label{mon}
t_{x_{1}}t_{x_{2}}t_{x_{3}} t_dt_{B_{2}}t_{e'}t_{x_{1}'}t_{x_{2}'}t_{x_{3}'}t_{d'}t_{B_{2}'}t_{C'}t_{\overline{e}}t_{\overline{x_1}}t_{\overline{x_2}}t_{\overline{x_3}}t_{\overline{d}}t_{\overline{B_2}}=1.
\end{eqnarray}

Let  $W$ be the positive factorization given in (\ref{mon})  and let $(X,f)$ be the corresponding genus-$3$ Lefschetz fibration which admits $18$ singular fibers with $(n,s)=(12,6)$. It is proved that the $4$-manifold $X_3$ is an exotic $\mathbb{C} P^{2}\#7\overline{\mathbb{C} P^{2}}$~\cite{b1}.

\subsection{Constructing a small nonhyperelliptic genus-$4$  Lefschetz fibration on a simply-connected $4$-manifold.}
\noindent

Consider the relation (\ref{mon}) in $\Mod_{3}^{2}$, which can be rewritten as 
\begin{eqnarray*}
W&=&Pt_{e'}P't_{C'}t_{\overline{e}}\overline{P}\\
&=&t_{\overline{e}}\overline{P}Pt_{e'}P't_{C'}\\
&=&t_{\overline{e}}t_{\overline{x_1}}t_{\overline{x_2}}t_{\overline{x_3}}t_{\overline{d}}t_{\overline{B_2}}t_{x_{1}}t_{x_{2}}t_{x_{3}} t_dt_{B_{2}}t_{e'}t_{x_{1}'}t_{x_{2}'}t_{x_{3}'}t_{d'}t_{B_{2}'}t_{C'}\\
&=&t_{\delta_1}t_{\delta_2}^{2},
\end{eqnarray*}
since the factorization $t_{\overline{e}}\overline{P}$ commutes with $t_{\delta_1}$ and $t_{\delta_2}$. We embed this relation into the surface $\Sigma_{4}^{2}$ so that the boundary parallel curve $\delta_2$ is mapped to $C''$, where the curves are as in Figure~\ref{TA4}. We also consider the Matsumoto's relation (\ref{mat2}) in $\Mod_{2}^{2}$
 \begin{eqnarray*}
  V_2=t_{C}^{2}t_{A_0} t_{A_1} t_{A_2} t_{B_0} t_{B_1} t_{B_2}=t_{\delta_1}t_{\delta_2},
 \end{eqnarray*}
and embed it into $\Sigma_{4}^{2}$ in such a way that the curves $ \delta_1=\partial \Sigma_{2}^{2}, A_i,B_i$ and $C$ are mapped to the curves $C',A_{i}'',B_{i}''$ and $C''$ for $i=0,1,2$, respectively, where the curves are depicted in Figure~\ref{TA4}. Thus, we get 
\begin{eqnarray*}
V_2W=\big(t_{C''}^{2} t_{A_0''} t_{A_1''} t_{A_2''}t_{B_{0}''} t_{B_{1}''} t_{B_{2}''}\big)\big(t_{\overline{e}}\overline{P}Pt_{e'}P't_{C'}\big)=\big(t_{C'}t_{\delta_2}\big) \big(t_{\delta_1}t_{C''}^{2}\big),
\end{eqnarray*}
which gives the relation
\begin{eqnarray*}
t_{C''}^{-2}\big(t_{C''}^{2} t_{A_0''} t_{A_1''} t_{A_2''}t_{B_{0}''} t_{B_{1}''} t_{B_{2}''}\big)\big(t_{\overline{e}}\overline{P}Pt_{e'}P't_{C'}\big)t_{C'}^{-1}=t_{\delta_1}t_{\delta_2},
\end{eqnarray*}
by the commutativity of Dehn twists $t_{C'}$, $t_{C''}$, $t_{\delta_1}$ and $t_{\delta_2}$. Therefore, we arrive the following relation in $\Mod_{4}^{2}$:
\begin{eqnarray*}
 t_{A_0''} t_{A_1''} t_{A_2''}t_{B_{0}''} t_{B_{1}''} t_{B_{2}''}t_{\overline{e}}\overline{P}Pt_{e'}P'=t_{\delta_1}t_{\delta_2},
\end{eqnarray*}
which gives the monodromy factorization for our first genus-$4$ Lefschetz fibration. Note that it admits two sections of self-intersection $(-1)$. By capping of the boundary components, we get the following factorization of identity in $\Mod_4$:
\begin{align}\label{mon4}
t_{A_0''} t_{A_1''} t_{A_2''}t_{B_{0}''} t_{B_{1}''} t_{B_{2}''}t_{\overline{e}}t_{\overline{x_1}}t_{\overline{x_2}}t_{\overline{x_3}}t_{\overline{d}}t_{\overline{B_2}}t_{x_{1}}t_{x_{2}}t_{x_{3}} t_dt_{B_{2}}t_{e'}t_{x_{1}'}t_{x_{2}'}t_{x_{3}'}t_{d'}t_{B_{2}'}=1.
\end{align}
\noindent

Let us denote the positive factorization (\ref{mon4}) by $W_1$ and the corresponding genus-$4$ Lefschetz fibration by $(X_1,f_1)$. It admits $23$ singular fibers with $(n,s)=(18,5)$. Observe that the Lefschetz fibration $(X_1,f_1)$ is nonhyperelliptic. Since it admits a section, the fundamental group 
$\pi_1(X_1)$ of the $4$-manifold $X_1$ is isomorphic to the quotient of $\pi_1(\Sigma_4)$ by the normal subgroup generated by the vanishing cycles of the Lefschetz fibration $(X_1,f_1)$.
\begin{figure}
\begin{center}
\scalebox{0.3}{\includegraphics{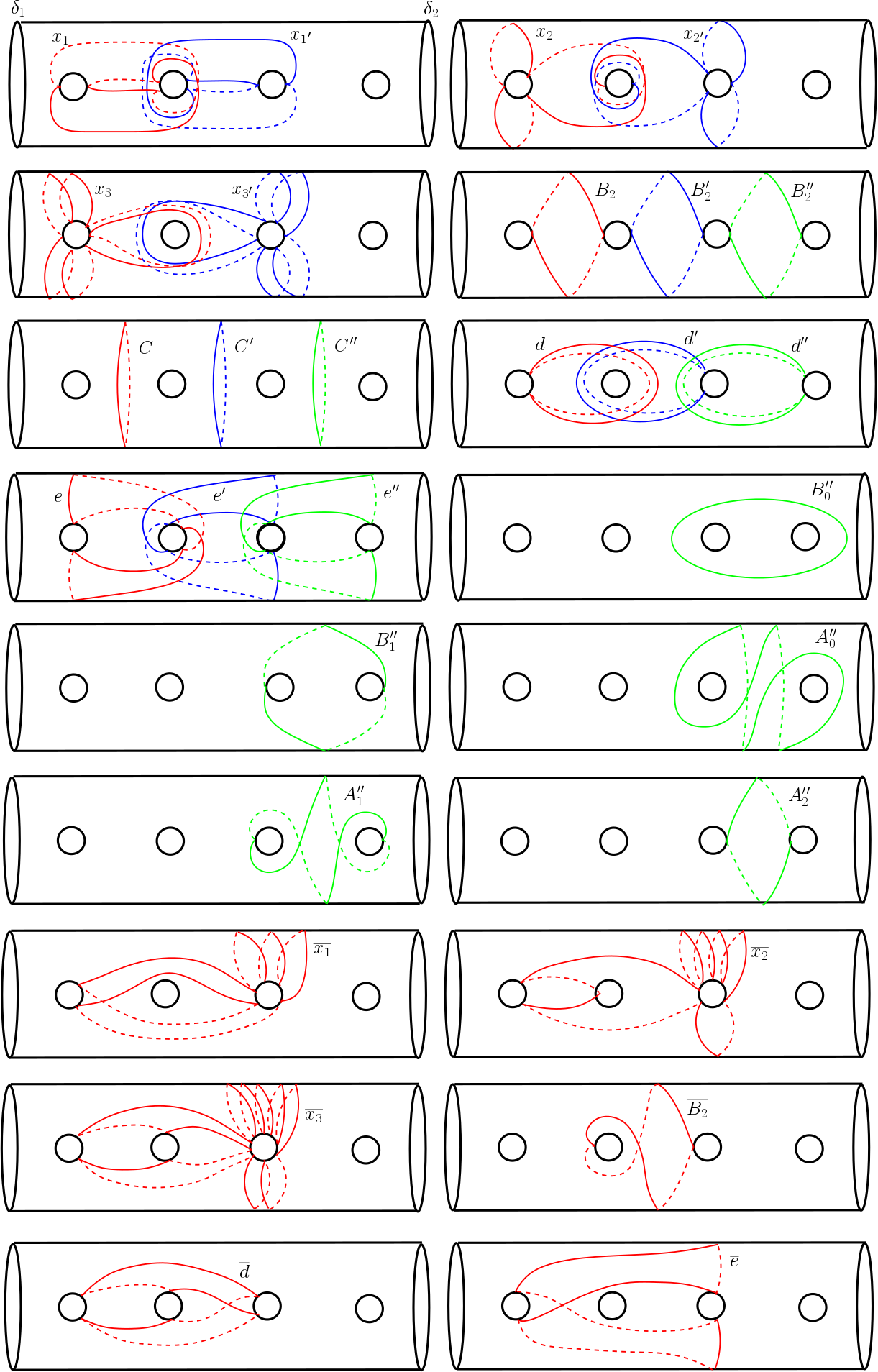}}
\caption{The curves on the surface $\Sigma_{4}^{2}$.}\label{TA4}
\end{center}
\end{figure}
\begin{figure}
\begin{center}
\scalebox{0.3}{\includegraphics{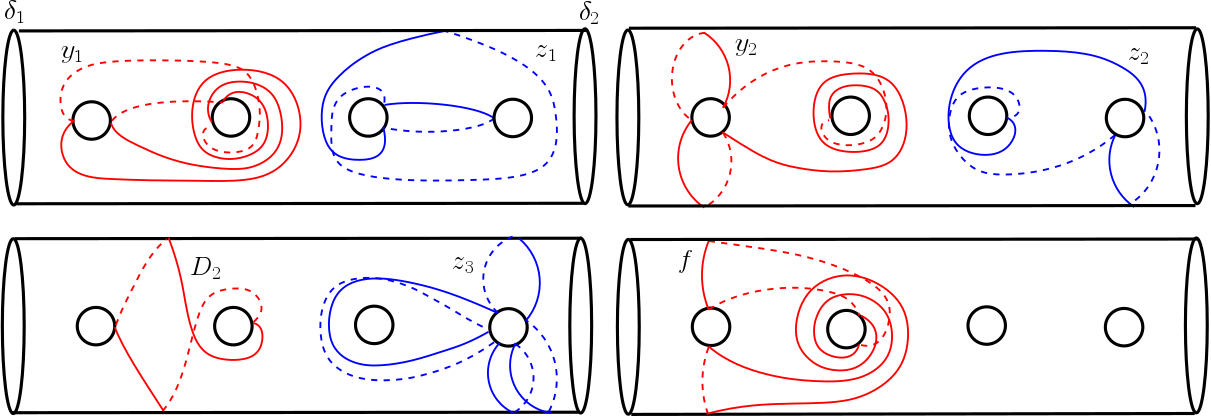}}
\caption{The curves on the surface $\Sigma_{4}^{2}$.}\label{TAH4}
\end{center}
\end{figure}

We now show that the $4$-manifold $X_1$ is simply-connected.

Consider the generators $a_i$, $b_i$ of $\pi(\Sigma_4)$ depicted in Figure~\ref{TA3} for $g=4$. Thus, $\pi_{1}(X_1)$ has a presentation with generators
 $a_1,a_2,a_3,a_4,b_1,b_2,b_3,b_4$ and with defining relations
\begin{eqnarray}
 \label{eq29}
b_4^{-1}b_3^{-1}b_2^{-1}b_1^{-1}(a_1b_1a_1^{-1})(a_2b_2a_2^{-1})(a_3b_3a_3^{-1})(a_4b_4a_4^{-1})=1.
\end{eqnarray}
  \[
  x_i=x_{i}'=\overline{x_i}=d=d'=\overline{d}= B_2=B_{2}'=\overline{B_2}=e'=\overline{e}=A_{i-1}''=B_{i-1}''=1, \ \ i=1,2, 3,
\]
where the curves shown in Figure~\ref{TA4}. One can get that $\pi_{1}(X_1)$ has the following relations (among many others):

\begin{eqnarray}
\label{eqv1}
x_1&=&b_1b_2a_2^{-1}a_1b_2a_2^{-1}a_1=1,\\
 \label{eqv2}
x_1'&=&b_2a_2^{-1}a_3b_3b_2a_2^{-1}a_3=1,\\
 \label{eqv3}
\overline{x_1}&=&a_1^{-1}a_2b_2a_2^{-1}a_3b_3a_3^{3}a_1^{-1}a_2b_2a_2^{-1}a_3=1,\\
\label{eqv4}
x_2&=&a_1^{2}b_1b_2^{2}a_2^{-1}a_1=1,\\
\label{eqv5}
d&=&b_2^{-1}a_1^{-1}a_2b_2a_2^{-1}a_1=1,\\
\label{eqv6}
d'&=&b_2a_3^{-1}a_2b_2^{-1}a_2^{-1}a_3=1,\\
\label{eqv7}
B_2&=&a_2^{-1}[a_1,b_1^{-1}]a_1^{-1}=1,\\
\label{eqv8}
B_2'&=&a_3^{-1}a_2b_2^{-1}a_2^{-1}[a_1,b_1^{-1}]b_2a_2^{-1}=1,\\
\label{eqv9}
\overline{B_2}&=&a_3^{-1}a_2b_2^{-1}a_2^{-1}[a_1,b_1^{-1}]b_2^{2}a_2^{-1}=1,\\
\label{eqv10}
B_{0}^{''}&=&b_3b_4=1,\\
\label{eqv11}
B_{1}^{''}&=&a_4^{-1}b_4^{-1}b_{3}^{-1}a_3^{-1}=1,\\
\label{eqv12}
B_{2}^{''}&=&a_{3}^{-1}[a_4,b_4]a_{4}^{-1}=1=1.
\end{eqnarray}
The relations (\ref{eqv8}) and (\ref{eqv9}) imply that $b_2=1$. Thus, it follows from the relations (\ref{eqv5}) and (\ref{eqv6}) that we get $[a_1,a_2]=[a_2,a_3]=1$, respectively. From these identities, the relations (\ref{eqv1}) and (\ref{eqv4}) give the following identites: 
\[
b_1=a_1^{-2}a_2^{2} \textrm{ and }b_1=a_1^{-3}a_2,
\]
which imply that $a_1=a_2^{-1}$ and so $b_1=a_1^{-4}$. Moreover, the relation (\ref{eqv7}) yields $[a_1,b_1]=1$. The relation (\ref{eqv10}) implies that $b_3=b_4^{-1}$. Thus, using the relations (\ref{eqv11})  and (\ref{eqv12}), we have $[a_4,b_4]=1$. By the relation (\ref{eq29}), one can get $[a_3,b_3]=1$. Using also the relations (\ref{eqv8}) and (\ref{eqv12}), the following identities hold:
\[
a_1=a_2^{-1}=a_3=a_4^{-1}.
\]
Moreover, by considering the relation (\ref{eqv2}), one can obtain that $b_3=a_1^{-4}$, which implies that $b_1=b_3=b_4^{-1}$. Hence, all relations show that $\pi_1(X_1)$ is generated by a single element, say $a_1$. However, the relation (\ref{eqv3}) become $b_3=a_1^{-3}$. Therefore, we conclude that $a_1=1$ by the fact that $b_3=a_1^{-4}$. Consequently, the fundamental group $\pi_1(X_1)$ is trivial. 
\begin{theorem}\label{thmex4}
The $4$-manifold $X_1$ is an exotic copy of $\mathbb{C} P^{2}\#8\overline{\mathbb{C} P^{2}}$.
\end{theorem}
\begin{proof}
The Euler characteristic $e(X_1)$ of $X_1$, is given by
\begin{eqnarray*}
e(X_1)&=&4-4g+n+s\\
&=&4-4(4)+18+5=11
\end{eqnarray*}
for $(n,s)=(18,5)$. We compute the signature $\sigma(X_1)$ of $X_1$ using Endo and Nagami's method. We obtain the monodromy of $X_1$ by breeding the Matsumoto's relation with the monodromy $W$, and then by cancelling the Dehn twists $t_{C''}^{2}$ and $t_{C'}$. Since the signature is additive, Theorem~\ref{signen} implies that $\sigma(X_1)$ is given by the sum of the signature of Matsumoto's relator, which is $-4$, the signature of relator coming from the monodromy $W_3$, which is $-6$ and finally the difference of signatures of three separating simple closed curve, which is $-3$. Hence, $\sigma(X_1)$ satisfies the following equality:
\begin{eqnarray*}
\sigma(X_1)=(-4)+(-6)-(-3)=-7.
\end{eqnarray*}
By the fact that  $X_1$ is simply-connected, it can be concluded that 
\begin{eqnarray*}
e(X_1)&=&2-2b_1(X_1)+b_{2}^{+}(X_1)+b_{2}^{-}(X_1)=11\\
&=&2+b_{2}^{+}(X_1)+b_{2}^{-}(X_1),
\end{eqnarray*}
and
\[
\sigma(X_1)=-7=b_{2}^{+}(X_1)-b_{2}^{-}(X_1),
\]
which give that $(b_2^{+}(X_1),b_2^{-}(X_1))=(1,8)$. By Freedman's classification, the $4$-manifold $X_1$ is homeomorphic to the rational surface $\mathbb{C} P^{2}\#8\overline{\mathbb{C} P^{2}}$. However the $4$-manifold $\mathbb{C} P^{2}\#8\overline{\mathbb{C} P^{2}}$ does not admit a genus-$4$ Lefschetz fibration ~\cite[Lemma $2$]{b1}. Hence it cannot be diffeomorphic to $X_1$, which implies that $X_1$ is an exotic copy of $\mathbb{C} P^{2}\#8\overline{\mathbb{C} P^{2}}$. 
\end{proof}
\subsection{Constructing a small hyperelliptic genus-$4$ Lefschetz fibration on a simply-connected $4$-manifold}
Consider the generalized Matsumoto's Lefschetz fibration for $g=4$ with the monodromy factorization (\ref{mat2})
\begin{eqnarray}\label{mat4}
V_4=t_{C'}^{2}(t_{\alpha_0}t_{\alpha_1}t_{\alpha_2}t_{\alpha_3}t_{\alpha_4})(t_{\beta_0}t_{\beta_1}t_{\beta_1}t_{\beta_3}t_{\beta_4})
=t_{\delta_1}t_{\delta_2},
\end{eqnarray}
where we denote the simple closed curves $A_i$, $B_i$ and $C$ by $\alpha_i$, $\beta_i$ and $C'$, respectively, to distuguish them from some of which appearing before.

We then embed the relation (\ref{eqw}), $T=t_{\delta_1}t_{\delta_2}$, in $\Mod_{2}^{2}$ into $\Mod_{4}^{2}$ in such a way that the boundary parallel curve $\delta_2$ shown in Figure~\ref{BK} is mapped to the curve $C'$ shown in Figure~\ref{TA4}. In this case, we get the relation $T=t_{\delta_1}t_{C'}$ in $\Mod_{4}^{2}$. The conjugation of this relation by $\phi=t_{b_2}$ gives the following relation:
\begin{eqnarray}
T^{\phi}&=&t_{\phi(e)}t_{\phi(x_1)}t_{\phi(x_2)}t_{\phi(x_3)}t_{\phi(d)}t_{\phi(B_2)}t_{\phi(C)}=t_{\phi(\delta_1)}t_{\phi(C')}\nonumber\\
&=&t_ft_{y_1}t_{y_2}t_{x_3}t_dt_{D_2}t_C=t_{\delta_1}t_{C'},\label{Tphi}
\end{eqnarray}
 where all curves containing the relation (\ref{Tphi}) are depicted in Figures~\ref{TA4} and~\ref{TAH4} (here we denote the curves $\phi(e)$, $\phi(x_1)$, $\phi(x_2)$ and $\phi(B_2)$ by $f$, $y_1$, $y_2$ and $D_2$, respectively).
 
In a similar way, we consider the relation (\ref{eqwtilde}), $T'=t_{\delta_1}t_{\delta_2}$, in $\Mod_{2}^{2}$, where the curves that appear in the factorization $T'$ are shown in Figures~\ref{BK} and~\ref{TC}. By conjugating this relation with the diffeomorphism $\psi=t_{a_2}^{-1}$, we get the relation
\begin{eqnarray*}
(T')^{\psi}&=&t_{\psi(e')}t_{\psi(x_1')}t_{\psi(x_2')}t_{\psi(x_3')}t_{\phi(d')}t_{\psi(B_2)}t_{\psi(C)}=t_{\psi(\delta_1)}t_{\psi(\delta_2)}\\
&=&t_{e'}t_{\psi(x_1')}t_{\psi(x_2')}t_{\psi(x_3')}t_{d'}t_{B_2}t_{C}=t_{\delta_1}t_{\delta_2}
\end{eqnarray*}
in $\Mod_{2}^{2}$. Now let us embed this relation into $\Mod_{4}^{2}$ so that the  curves $\delta_1$, $e',\psi(x_i')$, $d'$, $B_2$ and $C$ are mapped to the curves $C'$, $e''$, $z_i$, $d''$, $B_2''$ and $C''$, respectively, given in Figures~\ref{TA4} and \ref{TAH4}. Thus, we get the following relation in $\Mod_{4}^{2}$: 
\begin{eqnarray}\label{Tpsi}
t_{e''}t_{z_1}t_{z_2}t_{z_3}t_{d''}t_{B_2''}t_{C''}=t_{C'}t_{\delta_2}.
\end{eqnarray}
The relations (\ref{mat4}), (\ref{Tphi}) and  (\ref{Tpsi}) give rise to the following relation:
\begin{eqnarray*}
t_{C'}^{2}(t_{\alpha_0}t_{\alpha_1}t_{\alpha_2}t_{\alpha_3}t_{\alpha_4})(t_{\beta_0}t_{\beta_1}t_{\beta_1}t_{\beta_3}t_{\beta_4})(t_ft_{y_1}t_{y_2}t_{x_3}t_dt_{D_2}t_C)(t_{e''}t_{z_1}t_{z_2}t_{z_3}t_{d''}t_{B_2''}t_{C''})=(t_{\delta_1}t_{\delta_2})(t_{\delta_1}t_{C'})(t_{C'}t_{\delta_2}),
\end{eqnarray*}
which can be written as 
\begin{eqnarray*}
t_{C'}^{-2}t_{C'}^{2}(t_{\alpha_0}t_{\alpha_1}t_{\alpha_2}t_{\alpha_3}t_{\alpha_4})(t_{\beta_0}t_{\beta_1}t_{\beta_1}t_{\beta_3}t_{\beta_4})(t_ft_{y_1}t_{y_2}t_{x_3}t_dt_{D_2}t_C)(t_{e''}t_{z_1}t_{z_2}t_{z_3}t_{d''}t_{B_2''}t_{C''})=t_{\delta_1}^{2}t_{\delta_2}^{2}
\end{eqnarray*}
since $t_{C'}$, $t_{\delta_1}$ and $t_{\delta_2}$ all commute with each other. Therefore, we have the following relation in $\Mod_{4}^{2}$:
\begin{eqnarray*}
(t_{\alpha_0}t_{\alpha_1}t_{\alpha_2}t_{\alpha_3}t_{\alpha_4})(t_{\beta_0}t_{\beta_1}t_{\beta_1}t_{\beta_3}t_{\beta_4})(t_ft_{y_1}t_{y_2}t_{x_3}t_dt_{D_2}t_C)(t_{e''}t_{z_1}t_{z_2}t_{z_3}t_{d''}t_{B_2''}t_{C''})=t_{\delta_1}^{2}t_{\delta_2}^{2}.
\end{eqnarray*}
By capping off both boundary components $\delta_1$ and $\delta_2$, we get the following factorization of identity in $\Mod_4:$
\begin{eqnarray}\label{monh4}
(t_{\alpha_0}t_{\alpha_1}t_{\alpha_2}t_{\alpha_3}t_{\alpha_4})(t_{\beta_0}t_{\beta_1}t_{\beta_1}t_{\beta_3}t_{\beta_4})(t_ft_{y_1}t_{y_2}t_{x_3}t_dt_{D_2}t_C)(t_{e''}t_{z_1}t_{z_2}t_{z_3}t_{d''}t_{B_2''}t_{C''})=1.
\end{eqnarray}
Let $W_2$ denote the positive factorization (\ref{monh4}) and let $(X_2,f_2)$ be the genus-$4$ Lefschetz fibration with the monodromy $W_2$. It admits $24$ singular fibers with $(n,s)=(18,6)$. Since its vanishing cycles are invariant under the hyperelliptic involution $\iota$, the Lefschetz fibration $(X_2,f_2)$ is hyperelliptic. Moreover it admits two sections of $(-2)$ self-intersection.

We now compute the fundamental group $\pi_1(X_2)$ of the $4$-manifold $X_2$. Since the Lefschetz fibration $(X_2,f_2)$ admits a section, its fundamental group 
$\pi_1(X_2)$  is isomorphic to the quotient of $\pi_1(\Sigma_4)$ by the normal subgroup generated by its vanishing cycles.

Consider the generators $a_i$, $b_i$ of $\pi(\Sigma_4)$ shown in Figure~\ref{TA3}. Thus, $\pi_{1}(X_2)$ has a presentation with generators
 $a_1,a_2,a_3,a_4,b_1,b_2,b_3,b_4$ and with defining relations
\begin{eqnarray}
 \label{eq29}
b_4^{-1}b_3^{-1}b_2^{-1}b_1^{-1}(a_1b_1a_1^{-1})(a_2b_2a_2^{-1})(a_3b_3a_3^{-1})(a_4b_4a_4^{-1})=1.
\end{eqnarray}
  \[
  \alpha_i=\beta_{i}=z_j=y_1=y_2=x_3=f=d=D_2=C=e''=d''=B_2''=C''=1, \ \ i=0,\dots,4 \textrm{ and } j=1,2,3,
\]
where the curves shown in Figures~\ref{TA4} and \ref{TAH4}. One can get that $\pi_{1}(X_2)$ has the following relations (among many others):
\begin{eqnarray}
\beta_0 &=& b_1b_2b_3b_4 =1,\label{eqh1}\\
\beta_1 &=& a_1b_1b_2b_3b_4a_4=1, \label{eqh2}\\
\beta_2 &=& a_1b_2b_3b_4a_4b_4^{-1}=1,\label{eqh3}\\
\beta_3 &=& a_2b_2b_3[b_4,a_4]a_3=1,\label{eqh4}\\
\beta_4 &=&a_3^{-1}a_2b_2^{-1}a_2^{-1}[a_1,b_1^{-1}]b_2a_2^{-1}=1,\label{eqh5}\\
y_1&=&b_1b_{2}^{2}a_{2}^{-1}a_1b_{2}^{2}a_2^{-1}a_1, \label{eqh6}\\
D_2&=&b_2a_2^{-1}[a_1,b_1^{-1}]a_1^{-1}=1, \label{eqh7} \\
C&=&[a_1,b_1]=1,\label{eqhc}\\
z_1&=&b_3a_3^{-1}a_4b_4a_4^{-1}b_3a_3^{-1}a_4=1, \label{eqh8}\\
C''&=&[a_4,b_4]=1, \label{eqh9} \\
B_2''&=&a_4^{-1}a_3b_3^{-1}a_3^{-1}a_2b_2^{-1}a_2^{-1}[a_1,b_1^{-1}]b_2b_3a_3^{-1}=1. \label{eqh10}
\end{eqnarray}
The relations (\ref{eqh1}) and (\ref{eqh2}) imply that $a_1a_4=1$. Thus, we get $b_2b_3=1$ using the relation  (\ref{eqh9}) and (\ref{eqh3}). This gives the relations  $b_1b_4=1$ and $a_2a_3=1$ by the relations (\ref{eqh1}) and (\ref{eqh4}), respectively. We then have $[a_2,b_2]=1$ from (\ref{eqh5}) and (\ref{eqhc}). Since $a_2=a_3^{-1}$ and $b_2=b_3^{-1}$, we conclude that $[a_3,b_3]=1$. Together with the relation (\ref{eqh10}), we get $a_3a_4=1$. Hence we have $a_1=a_2^{-1}=a_3=a_4^{-1}$. This implies that $b_2=b_3=1$ using the relations (\ref{eqh7}) and (\ref{eqhc}). From this, we obtain the relations $b_1=a_1^{-4}$ and $b_1=a_1^{-3}$ using the relations  (\ref{eqh6}) and (\ref{eqh8}), respectively. This implies that $a_1=1$ and so $b_1=b_4=1$. We therefore get $\pi_{1}(X_2)=1$.

\begin{theorem}\label{thmexh4}
The $4$-manifold $X_2$ is an exotic copy of $\mathbb{C} P^{2}\#9\overline{\mathbb{C} P^{2}}$.
\end{theorem}
\begin{proof}
The Euler characteristic $e(X_2)$ of $X_2$ is  given by
\begin{eqnarray*}
e(X_2)&=&4-4g+n+s_1+s_2\\
&=&4-4(4)+18+6+0=12
\end{eqnarray*}
for $(n,s_1,s_2)=(18,6,0)$, where $s=s_1+s_2$. Since the Lefschetz fibration $(X_2,f_2)$ is hyperelliptic, we compute the signature $\sigma(X_2)$ of $X_2$ using the signature formula given in Lemma~\ref{lem2}. Thus, the signature $\sigma(X_2)$ is given by
\begin{eqnarray*}
\sigma(X_2)=\frac{1}{9}(-5n+3s_1+7s_2)=-8.
\end{eqnarray*}
It follows from $\pi_1(X_2)=1$ that one can conclude that 
\begin{eqnarray*}
e(X_2)&=&2-2b_1(X_2)+b_{2}^{+}(X_2)+b_{2}^{-}(X_2)\\
&=&2+b_{2}^{+}(X_2)+b_{2}^{-}(X_2)=12,
\end{eqnarray*}
and
\[
\sigma(X_2)=b_{2}^{+}(X_2)-b_{2}^{-}(X_2)=-8,
\]
which imply that $(b_2^{+}(X_2),b_2^{-}(X_2))=(1,9)$. By Freedman's classification, the $4$-manifold $X_2$ is homeomorphic to the $4$-manifold $\mathbb{C} P^{2}\#9\overline{\mathbb{C} P^{2}}$. However the rational surface $\mathbb{C} P^{2}\#9\overline{\mathbb{C} P^{2}}$ does not admit a genus-$4$ Lefschetz fibration ~\cite[Lemma $2$]{b1}. Thus it cannot be diffeomorphic to $X_2$. We therefore conclude that $X_2$ is an exotic $\mathbb{C} P^{2}\#9\overline{\mathbb{C} P^{2}}$. 
\end{proof}
\section{The minimal number of singular fibers in Lefschetz fibrations on simply-connected $4$-manifolds}\label{S4}
In this section, we examine the minimal number of singular fibers in Lefschetz fibrations on simply connected $4$-manifolds.  We remind that $N_g$ (respectively $M_g$) denotes the minimal number of singular fibers in all genus-$g$ (respectively hyperelliptic) Lefschetz fibratons on a simply connected $4$-manifold. Also, let us recall that $n$ and $s$ denote the number of non-separating and separating singular fibers in a genus-$g$ Lefschetz fibration, respectively. In the following lemma, based on the Cadavid's signature inequality (\ref{eq2}), we arrive a lower bound for $n$.
\begin{lemma}
\label{lemm1}
 Let $X$ be a simply-connected $4$-manifold that admits a genus-$g$ Lefschetz fibration over $\mathbb{S}^2$, then  $n\geq4g$.
\end{lemma}
\begin{proof}
It follows from $b_1(X)=0$ that
the Euler characteristic of $X$ is given by
\begin{eqnarray*}
e(X)&=&2-2b_1(X)+b_{2}^{+}(X)+b_{2}^{-}(X)\\
&=&2+b_{2}^{+}(X)+b_{2}^{-}(X).
\end{eqnarray*}
Hence the invariant $\chi_{h}(X)$ equals to the quotient
\begin{eqnarray*}
\chi_{h}(X)&=&\dfrac{e(X)+\sigma(X)}{4}\\
&=&\dfrac{(2+b_{2}^{+}(X)+b_{2}^{-}(X))+(b_{2}^{+}(X)-b_{2}^{-}(X))}{4}\\
&=&\dfrac{1+b_{2}^{+}(X)}{2}.
\end{eqnarray*}
Since the invariant $\chi_{h}(X)$ is an integer, we conclude that $\chi_{h}(X)\geq1$. On the other hand, the inequality (\ref{eq2}) implies that
\begin{eqnarray*}
1\leq \chi_{h}(X)&=&\dfrac{e(X)+\sigma(X)}{4}\\
&\leq&\dfrac{(4-4g+n+s)+(n-s-4g)}{4}\\
&=&\dfrac{2n-8g+4}{4},
\end{eqnarray*}
which gives the required inequality.
\end{proof}
\begin{remark}\label{rmk}
One can observe that any genus-$g\geq2$ hyperelliptic Lefschetz fibration on a simply-connected $4$-manifold $X$ must satisfy $n+s>4g$. Otherwise such a fibration admits $n+s=4g$ singular fibers (in this case $s=0$ by Lemma~\ref{lemm1}), then $(n,s)=(4g,0)$ does not satisfy the equation in Lemma~\ref{lem1}, which leads a contradiction. Hence, we conclude that $M_g\geq 4g+1$.
\end{remark}
\begin{remark} 
If there exists a genus-$g\geq2$ Lefschetz fibration on a simply-connected $4$-manifold $X$ with $n=4g$, then it follows from the proof of Lemma~\ref{lemm1} that $\chi_h(X)=1$. Therefore, $X$ satisfies $(b_{2}^{+}(X),b_{2}^{-}(X))=(1,1+s)$ and $(e(X),\sigma(X))=(4+s,-s)$. Now if $\sigma(X)=-s\not\equiv 0 \pmod{16}$, then $X$ is homeomorphic to $\mathbb{C} P^{2}\#(1+s)\overline{\mathbb{C} P^{2}}$ by Rokhlin's theorem. Also, if $0<s\leq 8$, then the $4$-manifold $X$ is an exotic copy of $\mathbb{C} P^{2}\#(1+s)\overline{\mathbb{C} P^{2}}$ by~\cite[Lemma $2$]{b1}.
\end{remark}


 It has been known that the minimal number of singular fibers in all torus Lefschetz fibrations is $12$. One can conclude that $N_1=M_1=12$ by the existence of torus Lefschetz fibrations with $12$ singular fibers on the elliptic surface $E(1)=\mathbb{C} P^{2}\#9\overline{\mathbb{C} P^{2}}$. Baykur and Korkmaz~\cite{bki} constructed a genus-$2$ Lefschetz fibration of type of $(8,6)$, so that its total space is an exotic copy of $\mathbb{C} P^{2}\#7\overline{\mathbb{C} P^{2}}$. By~\cite[Theorem $2$]{bki} one can conclude that it is the smallest simply-connected $4$-manifold which admits a genus-$2$ Lefschetz fibration. Thus, it follows immediately that 
 $M_2=N_2=14$. In the following lemma, we obtain the same result by using slightly different arguments than those
used in the proof of~\cite[Theorem $2$]{bki}.
 \begin{lemma}\label{lemn2}
$N_2=M_2=14$.
\end{lemma}
\begin{proof}
Suppose that there exists a genus-$2$ Lefschetz fibration on a simply-connected $4$-manifold $X$ with $n+s< 14$. Note that the Lefschetz fibration may have only type-$1$
separating vanishing cycles, $s=s_1$. By lemmata~\ref{lem2} and ~\ref{lem1} together with the inequality (\ref{eq2}), we get:
\begin{itemize}
\item$e(X)=n+s-4$,
\item$\sigma(X)=-\frac{1}{5}(3n+s)\leq n-s-8$ and 
\item$n+2s \equiv 0 \pmod{10}$.
\end{itemize}
We have also $n\geq8$ by Lemma~\ref{lemm1}. Thus, the possible values of $(n,s)$ are $(8,1)$ and $(10,0)$. On the other hand, by the proof of Lemma~\ref{lemm1}, the $4$-manifold $X$ must satisfy $ \chi_{h}(X) \geq 1$.
However, both values $(8,1)$ and $(10,0)$ of $(n,s)$ have $ \chi_{h}(X)=0$, which leads a contradiction. Hence, we conclude that $N_2=M_2=14$.
\end{proof}

Baykur constructed a genus-$3$ hyperelliptic Lefschetz fibration on a $4$-manifold which is an exotic $\mathbb{C} P^{2}\#7\overline{\mathbb{C} P^{2}}$ with $(n,s)=(12,6)$~\cite{b1}. Therefore, using also Lemma~\ref{lemm1} and Remark~\ref{rmk}, we get $12\leq N_3\leq18$ and $13\leq M_3\leq18$. For the number $M_3$, we have the following lemma:
\begin{lemma}\label{lemn3}
$M_3=18.$
\end{lemma}
\begin{proof} We shall use a similar proof as in Lemma~\ref{lemn2}.
Suppose that there exists a genus-$3$ hyperelliptic Lefschetz fibration on a simply-connected
$4$-manifold $X$ with $n+s <18$. The $4$-manifold $X$ may admit only separating vanishing cycles of type-$1$, so $s=s_1$. It follows from Lemmata~\ref{lem1} and ~\ref{lem2} and the inequality~\ref{eq2} that we have:
\begin{itemize}
\item$e(X)=n+s-8$,
\item$\sigma(X)=\frac{1}{5}(4n-s)\leq n-s-12$ and 
\item$n+12s \equiv 0 \pmod{28}$.
\end{itemize}
 Also, Lemma~\ref{lem1} implies that $n\geq12$. Thus, we have  only one possible value $(16,1)$ of $(n,s)$. On the other hand, by the proof of Lemma~\ref{lemm1}, the $4$-manifold $X$ must have $ \chi_{h}(X)\geq 1$.  However, the pair $(16,1)$ satisfies $ \chi_{h}(X)=0$, which gives a contradiction. Hence, $M_3=18$.
\end{proof}
We have constructed the nonhyperelliptic Lefschetz fibration $(X_1,f_1)$ of genus-$4$ with $23$ singular fibers whose total space is an exotic copy of $\mathbb{C} P^{2}\#8\overline{\mathbb{C} P^{2}}$ (see Theorem~\ref{thmex4}). We thus get $16\leq\;N_4\;\leq23$. Similarly, the genus-$4$ hyperelliptic Lefschetz fibration $(X_2,f_2)$ with $24$ singular fibers whose the total space is an exotic $\mathbb{C} P^{2}\#9\overline{\mathbb{C} P^{2}}$ (see Theorem~\ref{thmexh4}) gives an upper bound for the number $M_4$. Therefore, using the lower bound mentioned in Remark~\ref{rmk}, we get $17\leq M_4 \leq24$. The following lemma gives a better lower bound for the number $M_4$.
\begin{lemma}\label{lemm4}
$21\leq M_4\leq 24$.\end{lemma}
\begin{proof}
Suppose that there exist a genus-$4$ hyperelliptic Lefschetz fibration on a simply-connected $4$-manifold $Y$ with $n+s< 24$, where $s=s_1+s_2$ .By Lemmata~\ref{lem1} and ~\ref{lem2}, we have the following:
\begin{itemize}
\item$e(Y)=n+s_1+s_2-12$,
\item$\sigma(Y)=\frac{1}{9}(-5n+3s_1+7s_2)$ and 
\item$n+12s_1+4s_2 \equiv 0 \pmod{18}$.
\end{itemize}
It follows from Lemma~\ref{lem1} that $n\geq16$. From these, we obtain the possible six values $(16,1,2)$, $(16,0,5)$, $(16,4,2)$, $(18,3,0)$, $(18,2,3)$ and $(20,1,1)$ of $(n,s_1,s_2)$. The decompositions $(16,1,2)$, $(18,3,0)$ and $(20,1,1)$ satisfy $\chi_h(Y)=0$, a contradiction. By considering the possible values $(16,0,5)$, $(16,4,2)$ and $(18,2,3)$ of $(n,s_1,s_2)$, it can be concluded that $21\leq M_4 \leq 24$. 
\end{proof}

As we have already proved above, $M_2=14$ and $M_3=18$ (Lemmata~\ref{lemn2} and \ref{lemn3}). Moreover, the total spaces of corresponding hyperelliptic Lefschetz fibrations are exotic copies of the rational surface $\mathbb{C} P^{2}\#7\overline{\mathbb{C} P^{2}}$. Since we know that $\mathbb{C} P^{2}\#7\overline{\mathbb{C} P^{2}}$ does not admit a genus-$g$ Lefschetz fibration for $g\geq2$ ~\cite[Lemma $2$]{b1}, the following question appears naturally:
\begin{question}
Does there exist a genus-$g$ Lefschetz fibration whose total space is an exotic copy of  $\mathbb{C} P^{2}\#7\overline{\mathbb{C} P^{2}}$ for $g\geq4$?
\end{question}

If there exist such a fibration, then  it admits $4g+6$ singular fibers with $(n,s)=(4g,6)$, which implies that $4g \leq N_g \leq 4g+6$. Moreover, if it is also hyperelliptic, one can conclude that $4g+1\leq M_g \leq 4g+6$. Since $M_g=4g+6$ holds for $g=2$ and $3$, it is also natural to ask the following question:
\begin{question}
Is it true that $M_g=4g+6$ for $g\geq4$?
\end{question}

\end{document}